\documentclass[11pt]{amsart}

\usepackage{amsmath,amssymb,latexsym}
\usepackage{fullpage}
\usepackage{etoolbox}
\usepackage{enumitem}
\usepackage[sharp]{easylist}
\usepackage{color}
\usepackage[colorlinks,linkcolor=blue,citecolor=magenta]{hyperref}
\usepackage[colorinlistoftodos,bordercolor=orange,backgroundcolor=orange!20,linecolor=orange,textsize=scriptsize]{todonotes}
\usepackage{soul}
\usepackage{microtype}
\usepackage{yhmath}
\usepackage{mathtools}
\usepackage{tikz}
\usetikzlibrary{matrix,automata,arrows,positioning,calc,chains,fit,shapes}
\definecolor{myblue}{RGB}{80,80,160}
\definecolor{mygreen}{RGB}{80,160,80}

\newtheorem{theorem}{Theorem}[section]

\newtheorem{proposition}[theorem]{Proposition}
\newtheorem*{claim}{Claim}
\newtheorem{lemma}[theorem]{Lemma}
\newtheorem*{proposition-nonumber}{Proposition}

\theoremstyle{definition}

\def\bbR{\mathbb{R}}

\def\bbZ{\mathbb{Z}}
\def\bolda{\boldsymbol{a}}
\def\boldx{\boldsymbol{x}}
\def\boldy{\boldsymbol{y}}
\def\boldz{\boldsymbol{z}}

\def\calF{\mathcal{F}}
\def\leq{\leqslant}
\def\geq{\geqslant}

\begin{document} 

%\title{Simultaneous optimization for a load balancing problem}
%\title{Three criteria at once for a load balancing problem}
\title{Maximal workload, minimal workload, maximal workload difference: optimizing all criteria at once}

\author{S\'ebastien Deschamps}
\address{S. Deschamps,
Saint-Gobain Recherche, France}
\email{sebastien.deschamps2@saint-gobain.com}

\author{Fr\'ed\'eric Meunier}
\address{F. Meunier, CERMICS, \'Ecole des Ponts, France}
\email{frederic.meunier@enpc.fr}

\thanks{The second author is the corresponding author.}

\begin{abstract}
    In a simple model of assigning workers to tasks, every solution that minimizes the load difference between the most loaded worker and the least loaded one actually minimizes the maximal load and maximizes the minimal load. This can be seen as a consequence of standard results of optimization over polymatroids. We show that similar phenomena still occur in close models, simple to state, and that do not enjoy any polymatroid structure.
\end{abstract}



\maketitle 

\section{Introduction}\label{sec:intro}

Let $G=(U,W,E)$ be a bipartite graph. Assume given non-negative real numbers $d_u$ for $u \in U$. Let $X\coloneqq \{\boldx\in \bbR_{\geq 0}^E \colon \sum_{e\in\delta(u)} x_e = d_u \; \forall u \in U\}$. Here, $\delta(v)$ denotes the edges incident to a vertex $v$. For $\boldx \in X$, denote by $\ell^{\max}(\boldx)$ (resp.\ $\ell^{\min}(\boldx)$) the quantity $\max_{w\in W}\sum_{e \in \delta(w)} x_e$ (resp.\ $\min_{w \in W}\sum_{e \in \delta(w)} x_e$). Then we have this fact, which might look surprising at first glance. 

\begin{proposition}\label{prop:ini}
Over $X$, every $\boldx$ minimizing $\ell^{\max}(\cdot) - \ell^{\min}(\cdot)$ minimizes $\ell^{\max}(\cdot)$ and maximizes $\ell^{\min}(\cdot)$.
\end{proposition}

What might make it counter-intuitive is that a same solution optimizes three criteria at the same time, while these criteria do not share the same monotonicity. An application is the following. Let $U$ be tasks to be performed, $W$ be workers, and $d_u$ be the total time to be spent on performing task $u$ (``demand''). Each worker $w$ can only work on a subset of the tasks (the neighborhood of $w$ in $G$). Given $\boldx \in X$, the quantity $x_e$ is the time spent by the endpoint in $W$ on performing the endpoint in $U$. With this interpretation, $\sum_{e \in \delta(w)} x_e$ is the total time worked by $w \in W$ (``load''), and $\ell^{\max}(\boldx)$ (resp.\ $\ell^{\min}(\boldx)$) is the maximal (resp.\ minimal) time worked among the workers. Each of these criteria can be seen as privileging some fairness between the workers. So the proposition above translates into: {\em If you minimize the difference of working times between the most loaded worker and the least loaded one, then you also minimize the load of the most loaded one and maximize the load of the least loaded one.} 
 
Proposition~\ref{prop:ini} can easily be derived from standard results on lexicographically optimal bases of polymatroids (which go back to the works of Meggido~\cite{megiddo1974optimal} and Fujishighe~\cite{fujishige1980lexicographically}). For such results, the polymatroid structure looks crucial. Maybe more surprising is the fact that the proposition still holds if $d$ takes integer values, and $X$ is restricted to integer points; this is then a consequence of a recent theorem by Frank and Murota~\cite{frank2022decreasing}, who also propose polynomial algorithms to handle the related optimization problems~\cite{frank2022decreasing-algo}. Special cases were obtained before; see, e.g., the work by Harvey et al.~\cite{harvey_semi-matchings_2006}, which finds its motivation in load balancing as well.

In the present work, we show that Proposition~\ref{prop:ini} above can be kept even without the structure of a polymatroid. Assume given in addition vectors $\bolda^{v} \in \bbR_{>0}^{\delta(v)}$ attached to the vertices $v$ of $G$. For a point $\boldx$ and a subset $A$ of its indices, we denote by $\boldx_A$ the vector $(x_i)_{i\in A}$. We change now the definition of $X$ to $X^{\bolda} \coloneqq \{\boldx \in \bbR_{\geq 0}^E \colon \bolda^u \cdot \boldx_{\delta(u)} = d_u \; \forall u \in U\}$
and the definition of $\ell^{\max}(\boldx)$ and $\ell^{\min}(\boldx)$ to
$\ell^{\max}(\boldx) \coloneqq \max_{w\in W} \bolda^w \cdot \boldx_{\delta(w)}$ and $\ell^{\min}(\boldx) \coloneqq \min_{w\in W} \bolda^w \cdot \boldx_{\delta(w)}$. The starting setting is the special case when $\bolda^v$ is the all-one vector. Our first result is the following. 

\begin{theorem}\label{thm:main_a}
Suppose $\ell^{\max}(\boldx) > \ell^{\min}(\boldx)$ for all $\boldx \in X^{\bolda}$. Then over $X^{\bolda}$, every $\boldx$ minimizing $\ell^{\max}(\cdot) - \ell^{\min}(\cdot)$ minimizes $\ell^{\max}(\cdot)$ and maximizes $\ell^{\min}(\cdot)$.
\end{theorem}

By compactness, the existence of $\boldx \in X^{\bolda}$ minimizing $\ell^{\max}(\cdot) - \ell^{\min}(\cdot)$ is ensured as soon as $X^{\bolda}$ is non-empty. Theorem~\ref{thm:main_a} is a generalization of Proposition~\ref{prop:ini}, apart for unfortunate inputs where it is possible to have the quantities $\bolda^w \cdot \boldx_{\delta(w)}$ equal for all $w \in W$ (``it is possible to get all workers equally loaded''), something which is not expected to occur generically. There are actually examples showing that this condition is necessary. There are also examples showing that the theorem does not hold when $d$ takes integer values and $X^{\bolda}$ is restricted to integer points. These examples are described in Section~\ref{sec:ex}. The proof of Theorem~\ref{thm:main_a} is given in Section~\ref{sec:proof-a}. A short way to derive Proposition~\ref{prop:ini} from Theorem~\ref{thm:main_a} is also given there.

Our second result shows that when the non-negativity constraint is dropped, we get somehow the reverse situation where minimizing $\ell^{\max}(\cdot)$ or maximizing $\ell^{\min}(\cdot)$ actually minimizes $\ell^{\max}(\cdot) - \ell^{\min}(\cdot)$. In this case, we are able to consider even more general dependencies of the loads to the values $x_{uw}$. Assume given maps $f_v\colon \bbR^{\delta(v)} \rightarrow \bbR$ attached to the vertices $v$ of $G$. We extend the definitions of $\ell^{\max}(\cdot)$ and $\ell^{\min}(\cdot)$ to $\ell^{\max}(\boldx) \coloneqq \max_{w\in W} f_w(\boldx_{\delta(w)})$ and $\ell^{\min}(\boldx) \coloneqq \min_{w\in W} f_w(\boldx_{\delta(w)})$. Let $X^f \coloneqq \{\boldx \in \bbR^E \colon f_u(\boldx_{\delta(u)}) = d_u \; \forall u \in U\}$. Note that, contrary to what holds for $X$ and $X^{\bolda}$, the $\boldx$ are not constrained to have non-negative components. Consistently, the $d_u$ are no longer assumed to be non-negative (but this does not really matter since the problem and the theorem are invariant by translation on $\bbR$).

Let $A$ be a finite set. To ease the statement of the theorem, we define a certain set $\calF_A$ of maps $\bbR^A\rightarrow\bbR$. A map $f\colon \bbR^A\rightarrow\bbR$ belongs to $\calF_A$ if $f(\boldx_A)$ is an increasing self-bijection of $\bbR$ when restricted to any component $x_i$ (the components of $\boldx_{A\setminus\{i\}}$ being then fixed). Linear maps $\boldx_A \mapsto \bolda \cdot \boldx_A$ with $\bolda \in \bbR_{>0}^A$ (which are maps considered in Theorem~\ref{thm:main_a}) belong to $\calF_A$. Moreover, the set $\calF_A$ is stable by many binary operations (e.g., addition and maximum), which shows that it actually contains quite complicated maps.

\begin{theorem}\label{thm:main_F}
    Suppose that each $f_v$ belongs to $\calF_{\delta(v)}$ and that $G$ is connected. Then, for every $\boldx\in X^f$ such that $\ell^{\max}(\boldx)> \ell^{\min}(\boldx)$, there exists $\boldx'\in X^f$ such that $\ell^{\max}(\boldx) > \ell^{\max}(\boldx') = \ell^{\min}(\boldx') > \ell^{\min}(\boldx)$.
\end{theorem}

%To avoid any ambiguity, each component of a map $h\colon \bbR^A\rightarrow \bbR$ is an increasing self-bijection of $\bbR$ means the following: fixing all components but one of some $\boldx_A \in \bbR^A$ and making the non-fixed component vary from $-\infty$ to $+\infty$, the value of $h(\boldx_A)$ varies continuously from $-\infty$ to $+\infty$ in a strictly monotone way.

When we specialize the maps $f_v$ to $\boldx_{\delta(v)} \mapsto \bolda^v \cdot \boldx_{\delta(v)}$ with $\bolda^v \in \bbR_{>0}^{\delta(v)}$, we are exactly in the setting of Theorem~\ref{thm:main_a}, except for the non-negativity constraint. Theorem~\ref{thm:main_F} implies in particular that if there exists $\boldx \in X^f$ minimizing $\ell^{\max}(\boldx)$, then this $\boldx$ is such that $\ell^{\max}(\boldx) = \ell^{\min}(\boldx)$. This is thus the opposite phenomenon as that of Theorem~\ref{thm:main_a}. Besides, even if $X^f$ is non-empty as soon as no vertex in $U$ is isolated, the infimum of $\ell^{\max}(\boldx)$ over $X^f$ is actually not necessarily attained. An example is given in Section~\ref{sec:ex}. The proof of Theorem~\ref{thm:main_F} is given in Section~\ref{sec:proof-F}. 

The paper ends with a few comments and open questions (Section~\ref{sec:concl}).

\newpage

\subsection*{Acknowledgments} During a class given by the second author several years ago, a student, Quentin Ego, claimed with conviction that Proposition~\ref{prop:ini} above holds. The second author told him that such a result was impossible, promised a counter-example, and realized only a few hours later that the student was correct. Therefore, the authors are grateful to him and his intuition, which made them investigate further the topic.

\section{Counter-examples to possible extensions}\label{sec:ex}

\subsection{The conclusion of Theorem~\ref{thm:main_a} does not necessarily hold when \texorpdfstring{$\ell^{\max}(\boldx) = \ell^{\min}(\boldx)$}{} for some \texorpdfstring{$\boldx \in X^{\bolda}$}{}}

\begin{figure}
\begin{tikzpicture}[thick,
  every node/.style={draw,circle},
  fsnode/.style={fill=myblue},
  ssnode/.style={fill=mygreen},
  every fit/.style={ellipse,draw,inner sep=-0.4cm,text width=2cm, text height=2cm},
  -,
  shorten >= 3pt,
  shorten <= 3pt
]

% the vertices of U
\begin{scope}[start chain=going below,node distance=7mm]
\foreach \i in {1,2}
  \node[fsnode,on chain] (f\i) [label=left: $u_{\i}$] {};
\end{scope}
%\node[fsnode,on chain] (f6) [label=left: 6] {};

% the vertices of V
\begin{scope}[xshift=4cm,yshift=1.0cm,start chain=going below,node distance=7mm]
\foreach \i in {1,2,...,4}
  \node[ssnode,on chain] (s\i) [label=right: $w_{\i}$] {};
\end{scope}

% the set U
% \node [myblue,fit=(f1) (f4),label=left:$I$] (I) {};
% the set V
% \node [mygreen,fit=(s1) (s),label=right:$J$] {};

% \node[red] (base) at (0,0) {};

% the edges
\draw[red] (f1) -- (s2);
\draw[red] (f2) -- (s3);
\draw[red] (f2) -- (s4);
\draw (f1) -- (s1);
\draw (f2) -- (s2);
\draw (f1) -- (s3);

\end{tikzpicture}
\caption{\label{fig:example}The conclusion of Theorem~\ref{thm:main_a} does not necessarily hold if there exists $\boldx \in X^{\bolda}$ such that $\ell^{\max}(\boldx)=\ell^{\min}(\boldx)$.}
\end{figure}

Consider the graph of Figure~\ref{fig:example}. Set $d_{u_1}=21$ and $d_{u_2}=14$, and let $\bolda^u \in \bbR^{\delta(u)}$ be the all-one vector and $\bolda^w \in \bbR^{\delta(w)}$ be defined by
\[
a^w_e=\left\{\begin{array}{rl}
10 & \text{if $e$ is in red on the figure,} \\
1 & \text{if $e$ is in black on the figure.} \\
\end{array}
\right.
\]
 The point $(x^*_{u_1w_1},x^*_{u_1w_2},x^*_{u_1w_3},x^*_{u_2w_2},x^*_{u_2w_3},x^*_{u_2w_4})=(20,1,0,10,2,2)$
is in $X^{\bolda}$ and such that $\ell^{\max}(\boldx^*)=\ell^{\min}(\boldx^*)=20$. Yet, the point $(\bar x_{u_1w_1},\bar x_{u_1w_2},\bar x_{u_1w_3},\bar x_{u_2w_2},\bar x_{u_2w_3},\bar x_{u_2w_4})=(11,0,10,14,0,0)$ is in $X^{\bolda}$ and such that $\ell^{\max}(\boldx)=14$, which shows that even if $\boldx^*$ does minimize $\ell^{\max}(\cdot)-\ell^{\min}(\cdot)$, it does not minimize $\ell^{\max}(\cdot)$.

\subsection{Counter-example to the integral version of Theorem~\ref{thm:main_a}} \label{subsec:no-int}
Consider the graph of Figure~\ref{fig:example-int}. Set $d_{u_1}=9$ and $d_{u_2}=10$, and let $\bolda^u \in \bbR^{\delta(u)}$ be the all-one vector and $\bolda^w \in \bbR^{\delta(w)}$ be defined by
\[
a^w_e=\left\{\begin{array}{rl}
100 & \text{if $e$ is in red on the figure,} \\
1 & \text{if $e$ is in black on the figure.} \\
\end{array}
\right.
\] 

\begin{claim}
There exists an $\boldx\in X^{\bolda} \cap \bbZ^E$ such that $\ell^{\max}(\boldx) \leq 9$ and every such $\boldx$ satisfies $\ell^{\max}(\boldx) - \ell^{\min}(\boldx) \geq 4$. 
\end{claim}

\begin{proof}
The point $(x_{u_1w_1},x_{u_1w_2},x_{u_1w_3},x_{u_2w_1},x_{u_2w_2},x_{u_2w_3})=(9,0,0,0,5,5)$ is in $X^{\bolda}$ and is such that $\ell^{\max}(\boldx)=9$. This proves the first part of the claim. Now, note that if there is a red edge in the support of an $\boldx \in X^{\bolda} \cap \bbZ^E$, then $\ell^{\max}(\boldx) \geq 100$. Therefore, the set of all $\boldx \in X^{\bolda} \cap \bbZ^E$ such that $\ell^{\max}(\boldx)=9$ is $\{(9,0,0,0,k,10-k)\colon k=0,\ldots,10\}$. 
\end{proof}

Consider the point $(x_{u_1w_1},x_{u_1w_2},x_{u_1w_3},x_{u_2w_1},x_{u_2w_2},x_{u_2w_3})=(1,4,4,4,3,3)$. This point is such that $\ell^{\max}(\boldx) - \ell^{\min}(\boldx) \leq 2$. Together with the claim, this shows that, over $X^{\bolda} \cap \bbZ^E$, minimizing $\ell^{\max}(\cdot)-\ell^{\min}(\cdot)$ does not necessarily minimize $\ell^{\max}(\cdot)$, contrary to what happens when we minimize over the whole $X^{\bolda}$.

\begin{figure}
\begin{tikzpicture}[thick,
  every node/.style={draw,circle},
  fsnode/.style={fill=myblue},
  ssnode/.style={fill=mygreen},
  every fit/.style={ellipse,draw,inner sep=-0.4cm,text width=2cm, text height=2cm},
  -,
  shorten >= 3pt,
  shorten <= 3pt
]

% the vertices of U
\begin{scope}[start chain=going below,node distance=7mm]
\foreach \i in {1,2}
  \node[fsnode,on chain] (f\i) [label=left: $u_{\i}$] {};
\end{scope}
%\node[fsnode,on chain] (f6) [label=left: 6] {};

% the vertices of V
\begin{scope}[xshift=4cm,yshift=0.5cm,start chain=going below,node distance=7mm]
\foreach \i in {1,2,3}
  \node[ssnode,on chain] (s\i) [label=right: $w_{\i}$] {};
\end{scope}

% the set U
% \node [myblue,fit=(f1) (f4),label=left:$I$] (I) {};
% the set V
% \node [mygreen,fit=(s1) (s),label=right:$J$] {};

% \node[red] (base) at (0,0) {};

% the edges
\draw[red] (f1) -- (s2);
\draw[red] (f1) -- (s3);
\draw[red] (f2) -- (s1);
\draw (f1) -- (s1);
\draw (f2) -- (s2);
\draw (f2) -- (s3);

\end{tikzpicture}
\caption{\label{fig:example-int}Theorem~\ref{thm:main_a} does not hold as such for integers.}
\end{figure}

\subsection{Non-existence of an optimal solution in Theorem~\ref{thm:main_F}}

Consider the graph of Figure~\ref{fig:example_F}. Assume that $d_{u_1}=d_{u_2}=1$ and that
\begin{align*}
f_{u_1}(\boldx_{\delta(u_1)}) &\coloneqq x_{u_1w_1} + x_{u_1w_2} \, , & f_{u_2}(\boldx_{\delta(u_2)}) & \coloneqq x_{u_2w_1} + x_{u_2w_2} \, , \\ f_{w_1}(\boldx_{\delta(w_1)}) &\coloneqq x_{u_1w_1} + \frac 1 2 x_{u_2w_1}\, , & f_{w_2}(\boldx_{\delta(w_2)}) &\coloneqq x_{u_2w_2} + \frac 1 2 x_{u_1w_2} \, .
\end{align*}
Setting $x_{u_1w_1} = x_{u_2w_2} = t$ and $x_{u_1w_2} = x_{u_2w_1} = 1-t$ and letting $t$ go to $-\infty$ shows that $\inf_{x\in X^f}\ell^{\max}(\boldx) = -\infty$, while there is clearly no $\boldx$ in $X^f$ achieving this value.

\begin{figure}
\begin{tikzpicture}[thick,
  every node/.style={draw,circle},
  fsnode/.style={fill=myblue},
  ssnode/.style={fill=mygreen},
  every fit/.style={ellipse,draw,inner sep=-0.4cm,text width=2cm, text height=2cm},
  -,
  shorten >= 3pt,
  shorten <= 3pt
]

\begin{scope}[start chain=going below,node distance=7mm]
\foreach \i in {1,2}
  \node[fsnode,on chain] (f\i) [label=left: $u_{\i}$] {};
\end{scope}

\begin{scope}[xshift=4cm,start chain=going below,node distance=7mm]
\foreach \i in {1,2}
  \node[ssnode,on chain] (s\i) [label=right: $w_{\i}$] {};
\end{scope}

\draw (f1) -- (s2);
\draw (f1) -- (s1);
\draw (f2) -- (s1);
\draw (f2) -- (s2);

\end{tikzpicture}
\caption{\label{fig:example_F}There does not necessarily exist an $\boldx \in X^ f$ minimizing $\ell^{\max}(\boldx)$ in Theorem~\ref{thm:main_F}.}
\end{figure}

\section{Proof of Theorem~\ref{thm:main_a}}\label{sec:proof-a}

In this section, we prove Theorem~\ref{thm:main_a}. This requires a preliminary lemma. We finish the section with a short proof of Proposition~\ref{prop:ini}, relying on Theorem~\ref{thm:main_a}.

Let $\boldx\in X^{\bolda}$. We denote by $W^{\max}(\boldx)$ (resp.\ $W^{\min}(\boldx)$) the set of vertices $w$ in $W$ for which $\bolda^w \cdot \boldx_{\delta(w)} = \ell^{\max}(\boldx)$ (resp.\ $\bolda^w \cdot\boldx_{\delta(w)} =\ell^{\min}(\boldx)$). We also define $U^{\max}(\boldx)$ to be the set of vertices $u\in U$ for which there exists $w\in W^{\max}(\boldx)$ with $x_{uw}>0$.

\begin{lemma}\label{lem:claim}
Suppose $d_u > 0$ for at least one $u \in U$. Let $\boldx \in X^{\bolda}$. If $N(U^{\max}(\boldx))\neq W^{\max}(\boldx)$, then there exists $\boldx' \in X^{\bolda}$ such that
\begin{itemize}
    \item $|W^{\max}(\boldx')| \leq |W^{\max}(\boldx)|$,
    \item $\ell^{\max}(\boldx') \leq \ell^{\max}(\boldx)$,
    \item $\ell^{\min}(\boldx') \geq \ell^{\min}(\boldx)$,
\end{itemize}
with at least one of these inequalities being strict.
\end{lemma}

\begin{proof}
Suppose that $N(U^{\max}(\boldx))\neq W^{\max}(\boldx)$. Since $d_u >0$ for at least one $u$, each vertex $w \in W^{\max}(\boldx)$ is incident to at least one edge $uw$ with $x_{uw} > 0$, and thus $W^{\max}(\boldx)\subseteq N(U^{\max}(\boldx))$. Hence, there exist $u'\in U^{\max}(\boldx)$ and $w' \in W\setminus W^{\max}(\boldx)$ such that $u'w'\in E$. Still by definition, there exists $w''\in W^{\max}(\boldx)$ such that $x_{u'w''}>0$. Therefore, there exists $\boldx' \in X^{\bolda}$ very close to $\boldx$ with $x'_e = x_e$ for $e\notin\{u'w'', u'w'\}$, $x'_e < x_e$ for $e =u'w''$, and $x'_e > x_e$ for $e =u'w'$. The way $\boldx'$ is built implies that $\ell^{\max}(\boldx') \leq \ell^{\max}(\boldx)$ and $\ell^{\min}(\boldx') \geq \ell^{\min}(\boldx)$. Moreover, since $\bolda^{w''}\cdot\boldx'_{\delta(w'')} < \bolda^{w''}\cdot\boldx_{\delta(w'')}$, we have $\ell^{\max}(\boldx') < \ell^{\max}(\boldx)$ or $|W^{\max}(\boldx')|< |W^{\max}(\boldx)|$.
\end{proof}

\begin{proof}[Proof of Theorem~\ref{thm:main_a}]
We assume that $d_u > 0$ for at least one $u$ in $U$ since otherwise there is nothing to prove. Let $\boldx^*$ be minimizing $\ell^{\max}(\cdot) - \ell^{\min}(\cdot)$ over $X^{\bolda}$. We prove that $\boldx^*$ also minimizes $\ell^{\max}(\cdot)$. The proof that it also maximizes $\ell^{\min}(\cdot)$ is omitted since it follows exactly the same lines.

Applying Lemma~\ref{lem:claim} shows that there exists $\boldx' \in X^{\bolda}$ such that $|W^{\max}(\boldx')| \leq |W^{\max}(\boldx^*)|$ with $\ell^{\max}(\boldx') \leq \ell^{\max}(\boldx^*)$ and $\ell^{\min}(\boldx') \geq \ell^{\min}(\boldx^*)$. By optimality of $\boldx^*$, we have actually $|W^{\max}(\boldx')| < |W^{\max}(\boldx^*)|$, $\ell^{\max}(\boldx') = \ell^{\max}(\boldx^*)$, and $\ell^{\min}(\boldx') =\ell^{\min}(\boldx^*)$. Since 
$|W^{\max}(\boldx')|$ cannot become arbitrarily small, repeated applications of Lemma~\ref{lem:claim} leads then to the existence of an $\bar \boldx \in X^{\bolda}$ such that $N(U^{\max}(\bar \boldx)) = W^{\max}(\bar \boldx)$, with $\ell^{\max}(\bar \boldx) = \ell^{\max}(\boldx^*)$ and $\ell^{\min}(\bar\boldx) = \ell^{\min}(\boldx^*)$.

Since $\ell^{\max}(\bar\boldx) - \ell^{\min}(\bar\boldx) > 0$, the sets $W^{\max}(\bar\boldx)$ and $W^{\min}(\bar\boldx)$ are disjoint. Consider now a $\boldy^*$ minimizing $\ell^{\max}(\cdot)$ over $X^{\bolda}$ (which exists by compactness). Define the following solution for $t \in [0,1]$:
\[
z_e(t)= \left\{
\begin{array}{ll} ty_e^* + (1-t)\bar x_e & \text{if $e\in\delta(U^{\max}(\bar\boldx))$,} \\ [\jot] \bar x_e &  \text{if $e\notin\delta(U^{\max}(\bar\boldx))$.}
\end{array}
\right.
\]
The vector $\boldz(t)$ belongs to $X^{\bolda}$ for all $t \in [0,1]$. Suppose for a contradiction that $\ell^{\max}(\boldy^*) <\ell^{\max}(\bar \boldx)$. Then, by linearity, $\bolda^w\cdot\boldz(t)$ is decreasing on $[0,1]$ for $w \in W^{\max}(\bar \boldx)$, while it is constant for $w \notin W^{\max}(\bar \boldx)$. Since $\ell^{\max}(\bar \boldx)=\bolda^w\cdot\boldz(0)>\bolda^{w'}\cdot\boldz(0)$ for $w \in W ^{\max}(\bar \boldx)$ and $w'\notin W ^{\max}(\bar \boldx)$, there exists a small $\bar t >0$ for which we still have $\bolda^w\cdot\boldz(\bar t)>\bolda^{w'}\cdot\boldz(\bar t)$ for $w \in W ^{\max}(\bar \boldx)$ and $w'\notin W ^{\max}(\bar \boldx)$. On the other hand, we have $\bolda^{w'}\cdot\boldz(t) = \bolda^{w'}\cdot\bar\boldx$ for $w'\notin W ^{\max}(\bar \boldx)$. We have thus $\ell^{\max}(\boldz(\bar t)) - \ell^{\min}(\boldz(\bar t)) < \ell^{\max}(\bar \boldx) - \ell^{\min}(\bar\boldx)$, which is a contradiction. Therefore $\ell^{\max}(\boldy^*)=\ell^{\max}(\bar \boldx)$. Since $\ell^{\max}(\bar\boldx)=\ell^{\max}(\boldx^*)$, we get that $\boldx^*$ also minimizes $\ell^{\max}(\cdot)$ over $X^{\bolda}$. 
\end{proof}

Proposition~\ref{prop:ini} states that when $a_e^v=1$ for all vertices $v$ and all edges $e\in\delta(v)$, we do not need the assumption that $\ell^{\max}(\cdot) - \ell^{\min}(\cdot)$ is bounded away from $0$ on $X^{\bolda}$. Actually, in this special case, when there exists $\boldx \in X^{\bolda}$ such that $\ell^{\max}(\boldx) = \ell^{\min}(\boldx)$, a simple calculation leads to the desired conclusion.

\begin{proof}[Proof of Proposition~\ref{prop:ini}]
    Let $\boldx^* \in X$ be such that $\ell^{\max}(\boldx^*)-\ell^{\min}(\boldx^*)$ is minimal. If $\ell^{\max}(\boldx^*)-\ell^{\min}(\boldx^*) > 0$, then we apply Theorem~\ref{thm:main_a}. Suppose thus that $\ell^{\max}(\boldx^*)-\ell^{\min}(\boldx^*) = 0$. All $\bolda^w\cdot\boldx_{\delta(w)}^*$ are equal. Note that for every $\boldx\in X$, we have 
$\sum_{w \in W}\bolda^w\cdot\boldx_{\delta(w)} = \sum_{u \in U}\sum_{e\in\delta(u)} x_e=\sum_{u\in U}d_u$. This equality implies $\ell^{\max}(\boldx) \geq \frac 1 {|W|} \sum_{u\in U}d_u$ for every $\boldx\in X$, and implies $\bolda^w\cdot\boldx_{\delta(w)}^* =\frac 1 {|W|} \sum_{u\in U}d_u$ since the $\bolda^w\cdot\boldx_{\delta(w)}^*$ are all equal. The vector $\boldx^*$ also minimizes $\ell^{\max}(\cdot)$.    
\end{proof}

\section{Proof of Theorem~\ref{thm:main_F}}\label{sec:proof-F}

In this section, we prove Theorem~\ref{thm:main_F}. The general strategy consists in proving the result when $G$ is a tree---for which a slightly stronger result is actually proved (existence and uniqueness of the minimizer of $\ell^{\max}(\cdot)$ over $X^f$)---and to extend it to all connected bipartite graphs by working on spanning trees.

The proof requires two preliminary lemmas. The first lemma is an easy technical lemma about the maps in $\calF_A$. The second lemma shows that, up to one vertex in $W$, the values of the $\ell_w(\boldx_{\delta(w)})$ can be arbitrarily fixed. 

\begin{lemma}\label{lem:inv}
    Let $f \in \calF_A$ with $|A| \geq 2$ and let $j$ and $j'$ be two distinct elements in $A$. For every $\boldy \in \bbR^{A\setminus\{j,j'\}}$ and every $a \in \bbR$, there exists a continuous decreasing map $g\colon \bbR \rightarrow \bbR$, depending on $\boldy$ and $a$, such that $f(\boldy,x_j,x_{j'})=a$ if and only if $x_{j'} = g(x_j)$.
\end{lemma}

\begin{proof}
    Let $t \in \bbR$. The map $h_t\colon s \in \bbR\mapsto f(\boldy,t,s) \in \bbR$ is a continuous increasing map, which goes from $-\infty$ to $+\infty$. By the intermediate value theorem, there exists a unique $s$ for which the image of $h_t$ is $a$. We set $g(t)$ to that unique $s$. Since $f$ is componentwise increasing, we have that $g$ is decreasing.

    We prove now that $g$ is continuous. Let $(t_n)$ be a sequence of positive real numbers converging to some $\bar t$. By construction, we have $f(\boldy,t_n,g(t_n))=a$ for all $n$. The sequence $(g(t_n))$ remains bounded since otherwise there would exists an infinite sequence $n_1,n_2,\ldots$ such that the $|g(t_{n_i})|$ would go to $+\infty$, which is impossible because then $|f(\boldy,t_n,g(t_n))|$ would go to $+\infty$ as well. Let $\bar s$ be any accumulation point of $(g(t_n))$. By the continuity of $f$, we have $f(\boldy,\bar t,\bar s)=a$. By the uniqueness of the $s$ such that $h_{\bar t}(s)=a$, we have $g(\bar t)=\bar s$. The accumulation point is thus unique, and $\lim_ng(t_n)=g(\lim_nt_n)$.
\end{proof}

\begin{lemma}\label{lem:fixing}
    Assume that $G$ is a tree with at least one edge. Let $w_0 \in W$ and a collection $\ell_w \in \bbR$ for $w \in W \setminus \{w_0\}$. Then, there exists a unique $\boldx \in X^f$ such that $f_w(\boldx_{\delta(w)}) = \ell_w$ for all $w \in W \setminus \{w_0\}$.
\end{lemma}

\begin{proof}
    See $G$ as a tree rooted at $w_0$. The $d_u$ and $\ell_w$ attached to the leaves determine completely $x_e$ for the edges $e$ incident to the leaves. Recurse.
\end{proof}

We establish now the slightly stronger version of Theorem~\ref{thm:main_F} that holds when $G$ is a tree.

\begin{lemma}\label{lem:existence}
    Assume that $G$ is a tree with at least one edge. Then there exists a unique $\boldx^*$ minimizing $\ell^{\max}(\cdot)$ over $X^f$. Moreover, this $\boldx^*$ is such that $\ell^{\max}(\boldx^*)=\ell^{\min}(\boldx^*)$ and it maximizes $\ell^{\min}(\cdot)$ over $X^f$.
\end{lemma}

\begin{proof}
    Let $(\boldx^{(n)})$ be a sequence such that $\ell^{\max}(\boldx^{(n)})$ converges to its infimum. By the monotone subsequence theorem, we can choose the sequence so that for every $e \in E$, the sequence $(x_e^{(n)})$ is monotone and thus converges in $\bbR \cup \{-\infty,+\infty\}$. 

    \smallskip
    
    {\em Existence of $\boldx^*$.} Consider any vertex $v$. If $(x_e^{(n)})$ goes to $+\infty$ for an edge $e$ in $\delta(v)$, then $(x_{e'}^{(n)})$ has to go to $-\infty$ for an edge $e'\in\delta(v)\setminus\{e\}$ since $\left(f_v(\boldx_{\delta(v)}^{(n)})\right)$ remains bounded. Consider now a vertex $u \in U$. If $(x_e^{(n)})$ goes to $-\infty$ for an edge $e$ in $\delta(u)$, then $(x_{e'}^{(n)})$ has to go to $+\infty$ for an edge $e'\in\delta(u)\setminus\{e\}$ since $f_u(\boldx_{\delta(u)}^{(n)})$ is constant equal to $d_u$. We claim that this implies that no edge $e$ in $G$ is such that $\lim_n x_e^{(n)}\in\{-\infty,+\infty\}$. Indeed, pick a maximal path in $G$ whose limits (on the edges) alternate between $-\infty$ and $+\infty$. Its endpoints are necessarily in $W$, and thus the length of the path is even. On the other hand, the edges incident to the endpoints have limits equal to $-\infty$, and the alternation of the limits implies that the length of the path is odd; a contradiction. With the continuity of the $f_v$, the limit $\boldx^*$ of $(\boldx^{(n)})$ is a minimizer of $\ell^{\max}(\cdot)$ over $X^f$.

\smallskip

    {\em Equality $\ell^{\max}(\boldx^*)=\ell^{\min}(\boldx^*)$.} Among all minimizers $\boldx^*$, pick one with a minimal number of vertices $w \in W$ such that $f_w(\boldx^*_{\delta(w)}) = \ell^{\max}(\boldx^*)$. Suppose for a contradiction that there is a vertex $w' \in W$ such that $f_{w'}(\boldx^*_{\delta(w')}) < \ell^{\max}(\boldx^*)$. Consider a path from $w'$ to a vertex $w'' \in W$ with $f_{w''}(\boldx^*_{\delta(w'')}) = \ell^{\max}(\boldx^*)$, and denote its edges $e_1,e_2,\ldots,e_k$ and its vertices $v_0,v_1,\ldots,v_k$, from $w'$ to $w''$. Note that $k$ is even. We will now slightly increase the $x_{e_i}^*$ with $i$ odd and slightly decrease those with $i$ even, maintaining all $\ell_i \coloneqq f_{v_i}(\boldx_{\delta(v_i)}^*)$ for $i=1,\ldots,k-1$, which will lead to a contradiction because this new $\boldx^*$ would have fewer vertices $w \in W$ such that $f_w(\boldx^*_{\delta(w)}) = \ell^{\max}(\boldx^*)$.

    This slight modification of the $x_e$ along the path can be described precisely as follows. For $i \in [k-1]$, let $g_i$ be the map whose existence is ensured by Lemma~\ref{lem:inv} for $\boldy = \boldx_{\delta(v_i)^*\setminus\{e_i,e_{i+1}\}}$, $j = e_i$, $j' = e_{i+1}$, and $a \coloneqq \ell_i = f_{v_i}(\boldx_{\delta(v_i)}^*)$. Set $\varphi_i \coloneqq g_i \circ \cdots \circ g_1$. It is continuous and decreasing (resp.\ increasing) when $i$ is odd (resp.\ even). In particular, $\varphi_{k-1}$ is continuous and decreasing. Define 
    \[
    x_e^*(t) \coloneqq \left\{\begin{array}{rr} t & \text{when $e=e_1$,} \\ \varphi_i(t) & \text{when $e=e_{i+1}$ for $i \in [k-1]$,} \\ x_e^* & \text{otherwise.} \end{array}\right.
    \]
    Note that $x_e^*(\cdot)$ is continuous and that $f_v(\boldx_{\delta(v)}^*(t))$ does  not depend on $t$ except when $v=w'$ and $v=w''$. The set 
    $T=\{t\in\bbR \colon f_{w'}(\boldx_{\delta(w')}^*(t)) < f_{w''}(\boldx_{\delta(w'')}^*(t))\}$ is open by continuity, and non-empty since it contains $x_{e_1}^*$. Choose any $t \in T$ smaller than $x_{e_1}^*$. Since $f_{w''}(\boldx_{\delta(w'')}^*(\cdot))$ is decreasing, this $t$ is such that $f_u(\boldx_{\delta(u)}^*(t)) = d_u$ for all $u$, with $f_w(\boldx_{\delta(w)}^*(t)) \leq \ell^{\max}(\boldx^*)$ for all $w$, and with more vertices $w$ for which this inequality is strict; a contradiction.
    
\smallskip

    {\em Uniqueness.} Uniqueness is a consequence of Lemma~\ref{lem:fixing}, with $\ell_w = \ell^{\max}(\boldx^*)$ for all $w$ but an (arbitrary) $w_0$.

    \smallskip

    {\em The point $\boldx^*$ maximizes $\ell^{\min}(\cdot)$ over $X^f$.} We only sketch the proof because it follows the same lines as above, and moreover this property of $\boldx^*$ is not used elsewhere in the paper. Assume for a contradiction that there exists $\bar\boldx$ such that $\ell^{\min}(\bar\boldx)>\ell^{\max}(\boldx^*)$. Pick an arbitrary vertex $w_0$. Using a similar technique as for proving $\ell^{\max}(\boldx^*)=\ell^{\min}(\boldx^*)$ (changes along $w_0$-$w$ paths for all $w \neq w_0$), we build an $\boldx' \in X^f$ such that $f_w(\boldx'_{\delta(w)})=\ell^{\max}(\boldx^*)$ for all $w \in W\setminus\{w_0\}$ and $f_{w_0}(\boldx'_{\delta(w_0)})>\ell^{\max}(\boldx^*)$. This is not possible because the value of $f_w(\boldx_{\delta(w)})$ for all $w$ but one determines completely all the $x_e$ (Lemma~\ref{lem:fixing}).
\end{proof}

\begin{proof}[Proof of Theorem~\ref{thm:main_F}]
    We assume that $G$ is connected with at least one edge (otherwise there is nothing to prove). Let $\boldx \in X^f$. Take any spanning tree $G'$ of $G$. Letting $x_e$ being fixed for every edge $e$ not in $G'$, we apply Lemma~\ref{lem:existence} on $G'$ to get the existence of the desired $\boldx'$.
\end{proof}

\section{Concluding remarks}\label{sec:concl}

\subsection{Integer versions} We have emphasized in the introduction that Proposition~\ref{prop:ini} has an integer version. Would it possible to have integer versions of Theorem~\ref{thm:main_a} and~\ref{thm:main_F}? The counter-example of Section~\ref{subsec:no-int} shows that we cannot expect a generalization without extra conditions.

\subsection{Relation with classical results in topology} The proof technique of Theorem~\ref{thm:main_F} is topological by many regards. While writing the paper, the authors had the feeling that standard results from topology would not only simplify the proof but also show how to generalize it, e.g., with the $f_v$ being no necessarily componentwise increasing. The Poincaré--Miranda theorem~\cite{miranda1940osservazione,vrahatis1989short} was considered at some point of the writing as a good, yet unsuccessfully, candidate. The relations between the current work and with standard results from topology remain to elucidate.

\subsection{Algorithms} Both Theorems~\ref{thm:main_a} and~\ref{thm:main_F} have algorithmic counterparts: for Theorem~\ref{thm:main_a}, we are in the realm of linear programming; for Theorem~\ref{thm:main_F}, its proof can easily be turned into an algorithm (relying on binary search), which is polynomial under reasonable assumptions on the maps $f_v$.

%\subsection{Extension of Theorem~\ref{thm:main_F} to $\overline \bbR$}

%\begin{proposition}
%Suppose that every $f_v$ is a componentwise increasing self-bijection of $\bbR$ and that $G$ is connected. If $d_u = -\infty$ for at least one $u \in U$, then $\inf_{\boldx \in X^f}\ell^{\max}(\boldx) = -\infty$.
%\end{proposition}

\bibliographystyle{plain}
\bibliography{load_balancing}

\end{document}